\documentclass[11pt, onesided, reqno]{amsart}

\usepackage{amssymb}
\usepackage{amsmath}
\usepackage{color}
\usepackage{enumerate}
\usepackage{graphicx}

\evensidemargin\oddsidemargin

\newtheorem{theorem}{Theorem}

\newtheorem{lemma}[theorem]{Lemma}

\theoremstyle{definition}

\newtheorem{remark}[theorem]{Remark}

\newcommand{\eqnsection}{
\renewcommand{\theequation}{\thesection.\arabic{equation}}
    \makeatletter
    \csname  @addtoreset\endcsname{equation}{section}
    \makeatother}
\eqnsection

\def\e{\mathbf{e}}
\def\E{\mathbb{E}}

\def\N{\mathbb{N}}

\def\R{\mathbb{R}}

\def\Pb{\mathbb{P}}

\newcommand{\equi}{\mathop{\sim}\limits}
\def\={{\,\;\mathop{=}\limits^{\text{(law)}}\;\,}}

\def\qed{\hfill$\square$}

\allowdisplaybreaks

\makeatletter
\@namedef{subjclassname@2020}{%
  \textup{2020} Mathematics Subject Classification}
\makeatother

\begin{document}

\title[]{Extreme values of critical and subcritical branching stable processes with positive jumps}
\author[Christophe Profeta]{Christophe Profeta}

\address{
Universit\'e Paris-Saclay, CNRS, Univ Evry, Laboratoire de Math\'ematiques et Mod\'elisation d'Evry, 91037, Evry-Courcouronnes, France.
 {\em Email} : {\tt christophe.profeta@univ-evry.fr}
  }



\begin{abstract} 
We consider a branching stable process with positive jumps, i.e. a continuous-time branching process in which the particles evolve independently as stable L\'evy processes with positive jumps. Assuming the branching mechanism is critical or subcritical, we compute the asymptotics of the maximum location ever reached by a particle of the process.

\end{abstract}

\maketitle

\section{Statement of the main result}

\subsection{Introduction}
We consider a one-dimensional branching stable L\'evy process. It is a continuous-time particle system in which individuals move according to independent $\alpha$-stable L\'evy processes, and split at exponential times into a random number of children. \\

\noindent
More precisely, the process starts at time $t = 0$ with a single particle located at the origin.

\begin{enumerate}
\item When not branching, each particle moves independently as a strictly $\alpha$-stable L\'evy process $L$ with positive jumps. We refer to Bertoin \cite[Chapter VIII]{Ber} for an overview of such process. In particular, the existence of positive jumps implies that the scaling parameter $\alpha$ and the skew parameter $\beta$ satisfy the conditions :
$$\alpha \in (0,1)\cup(1,2) \text{ and } \beta \in (-1,1]\qquad \text{ or }\qquad \alpha=1\text{ and } \beta=0.$$
\item Each particle lives for an exponentially distributed time of parameter 1, independently of the others. When it dies, it splits into a random number of children with distribution $\boldsymbol{p}=(p_n)_{n\geq0}$. We assume that the distribution $\boldsymbol{p}$ is non trivial (i.e. $p_1\neq1$) and admits moments of order at least 3, i.e. $\E[\boldsymbol{p}^3]<+\infty$.\\
\end{enumerate}

\noindent
Such process may be constructed by first running a standard continuous-time Markov branching process $Z$ (see for instance \cite[Chapter III]{AtNe}), and then running independent $\alpha$-stable L\'evy processes $(L^{(i)})$ along the edges. With this notation, for each $t>0$, the number of particles alive at time $t$ is thus given by $Z(t)$, and their locations by 
\begin{equation}\label{eq:Z}
\left\{L_t^{(1)}, \ldots, L_t^{(Z(t))} \right\}.
\end{equation}

\noindent
It is classic that when $\E[\boldsymbol{p}]\leq 1$,   the process will go extinct in finite time with probability one. As a consequence, one may define the overall maximum ${\bf M}_{\alpha, \beta}$ ever attained by one of the particle. The main result of the paper is the computation of the asymptotics of its tail distribution :
$$u(x) = \Pb\left({\bf M}_{\alpha, \beta}\geq x\right).$$

\noindent

\begin{theorem}\label{theo:1}
Let  $\kappa_{\alpha, \beta}>0$ be the constant such that :
$$\Pb\left(L_1\geq x\right) \equi_{x\rightarrow +\infty} \kappa_{\alpha, \beta}\;x^{-\alpha}.
$$

\begin{enumerate}[$i)$]
\item Assume that $\E[\boldsymbol{p}]<1$. The asymptotics of ${\bf M}_{\alpha, \beta}$ is given by 
$$\Pb({\bf M}_{\alpha, \beta}\geq x) \equi_{x\rightarrow +\infty}  \frac{\kappa_{\alpha, \beta}}{1-\E[\boldsymbol{p}]}  x^{-\alpha}.$$
\item Assume that $\E[\boldsymbol{p}]=1$. The asymptotics of ${\bf M}_{\alpha, \beta}$ is given by 
$$\Pb({\bf M}_{\alpha, \beta}\geq x) \equi_{x\rightarrow +\infty}  \sqrt{\frac{2\kappa_{\alpha, \beta}}{\sigma^2}}  x^{-\alpha/2}$$
where
$$\sigma^2 =\emph{Var}(\boldsymbol{p})= \sum_{n=2}^{+\infty} n(n-1)p_n. $$
\end{enumerate}
\end{theorem}

\begin{remark}\label{rem:1}
The constant $\kappa_{\alpha,\beta}$ may be computed explicitly, but depends on the normalization chosen for  $L$. For instance, following Sato \cite[p.88]{Sat}, if the characteristic exponent of $L$ is given by :
\begin{equation}\label{eq:Lt}
 \ln\left(\E\left[e^{\textbf{i} \lambda L_t}\right]\right)= \begin{cases}
 - c_{\alpha, \beta} |\lambda|^\alpha \left(1- \textbf{i}\beta \tan\left(\frac{\pi \alpha}{2}\right)\text{sgn}(\lambda)\right)&\quad \text{for }\alpha\neq 1\\
 -  |\lambda| &\quad \text{for }\alpha= 1\\ 
 \end{cases}
 \end{equation}
  with $c_{\alpha, \beta} = \cos\left(\frac{\pi \beta}{2} \min(\alpha,\;2-\alpha)\right)$, then,  
$$\kappa_{\alpha, \beta} =\begin{cases}
 \frac{1}{\pi} \Gamma(\alpha) \sin\left(\frac{\pi \alpha}{2}(1+\beta)\right)&\quad \text{if }\alpha <1\\
 \frac{1}{\pi} &\quad \text{if }\alpha =1\\
  \frac{1}{\pi} \Gamma(\alpha) \sin\left(\frac{\pi}{2}(\alpha +\alpha\beta-2\beta)\right)&\quad \text{if }\alpha >1\\
\end{cases}$$
\end{remark}

The occurrence of $\sigma^2$ is a classic feature of such asymptotics, and was already observed by Fleischman \& Sawyer \cite{FlSa} in the case of Branching Brownian motion, or by Lalley \& Shao \cite{LaSh1} in the case of Branching Random Walks.\\

In the symmetric case (i.e. $\beta=0$) and when $p_0=p_2=\frac{1}{2}$, Theorem \ref{theo:1} was first obtained by Lalley \& Shao \cite{LaSh2}, who prove that
\begin{equation}\label{eq:LaSh}
\Pb\left({\bf M}_{\alpha,0}\geq x\right) \equi_{x\rightarrow +\infty} \sqrt{\frac{2}{\alpha}} x^{-\alpha/2}
\end{equation}
when choosing the normalization 
$$\E\left[e^{\textbf{i}\lambda L_1}\right] = \exp\left(-\int_\R (1-e^{\textbf{i}\lambda x}) \frac{dx}{|x|^{\alpha+1}} \right)=\exp\left(-\frac{\pi}{\Gamma(\alpha+1)\sin\left(\frac{\pi \alpha}{2}\right)} |\lambda|^{\alpha}\right).$$
In this case,  from Remark \ref{rem:1} and the recurrence formula for the Gamma function, the constant $\kappa_{\alpha,0}$ equals
$$\kappa_{\alpha,0} = \frac{\pi}{\Gamma(\alpha+1)\sin\left(\frac{\pi \alpha}{2}\right)}\times \frac{1}{\pi} \Gamma(\alpha) \sin\left(\frac{\pi \alpha}{2}\right)  = \frac{1}{\alpha}$$
hence (\ref{eq:LaSh}) agrees with Theorem \ref{theo:1} since $\sigma^2=1$ for this binary branching mechanism. \\

The starting point of \cite{LaSh2} was to show that $u$ is the  solution of a pseudo-differential equation involving the generator of the symmetric stable L\'evy process. This in turn allows to obtain a Feynman-Kac representation of $u$, and the authors then deduce the asymptotics of Theorem \ref{theo:1} after a careful analysis of the jumps of the underlying stable L\'evy process. We shall propose here another approach and rather work with an integral equation.

\subsection{An integral equation for $u$}
We start by writing down the integral equation satisfied by the function $u$. Let us denote by 
$$S_t = \sup_{s\in [0,t]}L_s$$
the running supremum of the stable L\'evy process $L$, and let $\e$ be a standard exponential random variable of parameter 1, independent from $L$.
\begin{lemma}
The function $u$ is a solution of the  integral equation :
\begin{equation}\label{eq:u}
u(x) =\Pb\left(S_\e\geq x\right)
+ \E\left[1_{\{S_\e<x\}}u(x-L_\e)\right]- \Phi_0(x)+ \Phi_R(x)
\end{equation}
where the main term $\Phi_0$ is given by 
$$\Phi_0 (x) = (1- \E[\boldsymbol{p}])  \E\left[1_{\{S_\e<x\}}u(x-L_\e)\right] +  \frac{1}{2}\E\left[\boldsymbol{p}^2-\boldsymbol{p}\right]\E\left[1_{\{S_\e<x\}}u^2(x-L_\e)\right]$$
and the remainder $\Phi_R$ satisfies the bounds 
\begin{equation}\label{eq:phiR}
0\leq \Phi_R(x)  \leq \E[\boldsymbol{p}^3]  \E\left[1_{\{S_\e<x\}}u^3(x-L_\e)\right] .
\end{equation}
\end{lemma}

\begin{proof}
We start by applying the Markov property at the first branching event :
\begin{align*}
\Pb({\bf M}_{\alpha, \beta}< x) &= p_0 \Pb\left(S_\e< x\right) + \sum_{n=1}^{+\infty} p_n\, \Pb\left(S_\e<x,\; L_\e+{\bf M}_{\alpha, \beta}^{(1)}<x,\ldots,  L_\e+{\bf M}_{\alpha, \beta}^{(n)}<x \right)
\end{align*}
where the  random variables $({\bf M}^{(n)}_{\alpha, \beta})_{n\in \N}$ are independent copies of ${\bf M}_{\alpha, \beta}$, which are also independent of the pair $(L_\e, S_\e)$.
As a consequence, we obtain the integral equation :
\begin{equation}\label{eq:u0}
1-u(x) = p_0 \Pb\left(S_\e< x\right)+ \sum_{n=1}^{+\infty} p_n \, \E\left[1_{\{S_\e<x\}}\; (1- u(x-L_\e))^n\right].
\end{equation}
Plugging the Taylor expansion with integral remainder 
$$(1-u)^n = 1 - n u + \frac{n(n-1)}{2} u^2 -  \frac{n(n-1)(n-2)}{6} u^3  \int_0^1 (1-u t)^{n-3} (1-t)^2 dt$$
in (\ref{eq:u0}), we deduce that 
$$
u(x) =\Pb\left(S_\e\geq x\right)
+ \E\left[1_{\{S_\e<x\}}u(x-L_\e)\right]- \Phi_0(x)+ \Phi_R(x)
$$
where the remainder $\Phi_R$ equals :
$$\Phi_R(x) = \sum_{n\geq 3}p_n   \frac{n(n-1)(n-2)}{6} \int_0^1  \E\left[1_{\{S_\e<x\}}u^3(x-L_\e) (1-u(x-L_\e) t)^{n-3} \right] (1-t)^2 dt  .  $$
Since $0\leq u(x)\leq 1$ for any $x\geq0$, the upper bound for $\Phi_R$ is obtained by bounding the term to the power $n-3$ by 1.

\end{proof}

\begin{remark}
Several terms in the equation (\ref{eq:u}) satisfied by $u$ look like convolutions products. This will lead us to work with Laplace transforms and we thus set for a positive function $f$ :
$$\mathcal{L}[f](\lambda) = \int_0^{+\infty} e^{-\lambda x} f(x) dx. $$
We shall repeatedly use in the following the  standard Karamata's Tauberian theorem (see for instance Korevaar \cite[Theorem 8.1]{Kor}) which states that for $\gamma\geq 0$ :
\begin{equation}
\mathcal{L}[f](\lambda) \mathop{\sim}\limits_{\lambda\rightarrow0}\frac{1}{\lambda^\gamma}h\left(\frac{1}{\lambda}\right)
\quad\Longleftrightarrow\quad \int_0^xf(z)dz\mathop{\sim}\limits_{x\rightarrow+\infty} \frac{1}{\Gamma(1+\gamma)}x^{\gamma}h(x),
\end{equation}
where $h$ is a slowly varying function.\\

\end{remark}

\begin{remark}
It may be noted that Equation (\ref{eq:u}) involves the distributions of $L_\e$ and $S_\e$. The key observation is the following equivalence of asymptotics for strictly $\alpha$-stable L\'evy processes with positive jumps, see Bertoin \cite[p.221]{Ber} :
\begin{equation}\label{eq:asympLS}
\Pb\left(L_1\geq x\right) \equi_{x\rightarrow +\infty}\Pb\left(S_1\geq x\right) \equi_{x\rightarrow +\infty} \kappa_{\alpha, \beta}\;x^{-\alpha}.
\end{equation}
We shall indeed prove, through Laplace transforms, that as $x\rightarrow +\infty$ :
$$\Pb\left(L_1\geq x\right) \leq (1- \E[\boldsymbol{p}]) u(x) +  \frac{1}{2}\E\left[\boldsymbol{p}^2-\boldsymbol{p}\right]u^2(x) \leq \Pb\left(S_1\geq x\right)$$
and the result will follow from the equivalence of both asymptotics.
\end{remark}

The remainder of the paper is devoted to the proof of Theorem \ref{theo:1} : Section \ref{sec:2} is dedicated to the case $\alpha \in (0,1]$ and Section \ref{sec:3} to the case $\alpha\in(1,2)$. The general idea of the proof is the same in both cases, and is composed of three steps. We shall first write down some general inequalities involving Laplace transforms, then apply these inequalities to Equation (\ref{eq:u}) and finally pass to the limit and apply Karamata's tauberian theorem. The main difference between both cases is the existence of the first moment of $L_\e$ when $\alpha\in(1,2)$. This will require us to make some extra computations, in order to remove the "first order" terms.

\section{The case $0<\alpha\leq 1$}\label{sec:2}

\noindent

\subsection{Preliminary lemma}
We start by writing some general bounds for the Laplace transform of the terms appearing in Equation (\ref{eq:u}). In the forthcoming proofs, we will  frequently use the case $f=u$ and $f=u^2$. To simplify the notation, we set :
$$\eta_\alpha(\lambda) =
\begin{cases}
\Gamma(1-\alpha) \lambda^{\alpha-1}& \text{if }\alpha<1\\
\displaystyle -\ln(\lambda)& \text{if }\alpha=1
\end{cases}
$$
and $L_\e^+ = \max(0, L_\e)$. Note that using (\ref{eq:asympLS}), the Tauberian theorem and the scaling property to remove the independent exponential random variable $\e$, we have  the asymptotics
\begin{equation}\label{eq:SeLe}
\frac{1-\E\left[e^{-\lambda S_\e}\right]}{\lambda} \equi_{\lambda\downarrow 0} \frac{1-\E\left[e^{-\lambda L_\e^+}\right]}{\lambda} \equi_{\lambda\downarrow 0}\kappa_{\alpha,\beta}\,\eta_\alpha(\lambda).
\end{equation}
\begin{lemma}\label{lem:f}
Assume that $\alpha\in(0,1]$ and let $f:[0,+\infty)\rightarrow[0,+\infty)$ be a positive and decreasing function. The following inequalities holds :
\begin{enumerate}[$i)$]
 \item Upper bound :
$$ \int_0^{+\infty}  e^{-\lambda x}  \E\left[1_{\{S_\e<x\}}f(x-L_\e)\right] dx \leq\E\left[e^{-\lambda S_\e}\right] \mathcal{L}[f](\lambda)$$
\item Lower bound :
\begin{multline*}
 \int_0^{+\infty}  e^{-\lambda x}  \E\left[1_{\{S_\e<x\}}f(x-L_\e)\right]dx\\
 \geq  \E\left[e^{-\lambda L_\e^+}\right]\mathcal{L}[f](\lambda)+f(0)\frac{\E\left[e^{-\lambda S_\e}\right]-\E\left[e^{-\lambda L_\e^+}\right]}{\lambda}   - \E\left[1_{\{L_\e<0\}} \int_{0}^{-L_\e} e^{-\lambda z  } f(z)dz\right]
\end{multline*}
\item  Assume that $\lim\limits_{x\rightarrow+\infty}f(x)=0$, then  \\
$$\lim_{\lambda\downarrow0} \frac{1}{ \eta_\alpha(\lambda)} \E\left[1_{\{L_\e<0\}} \int_{0}^{-L_\e} e^{-\lambda z  } f(z)dz\right]=0.$$
\end{enumerate}
\end{lemma}

\begin{proof}
The upper bound $i)$ is a direct consequence of the a.s. inequality $L_\e \leq S_\e$. Indeed, since $f$ is decreasing, the Laplace transform of the convolution product yields 
$$ \int_0^{+\infty}  e^{-\lambda x}  \E\left[1_{\{S_\e<x\}}f(x-L_\e)\right] dx \leq  \int_0^{+\infty}  e^{-\lambda x}  \E\left[1_{\{S_\e<x\}}f(x-S_\e)\right]dx = \E\left[e^{-\lambda S_\e}\right]\mathcal{L}[f](\lambda).$$
For the lower bound $ii)$, still using that $f$ is decreasing and $L_\e \leq S_\e$ a.s., 
\begin{align*}
&\int_0^{+\infty}e^{-\lambda x}  \E\left[1_{\{S_\e<x\}}f(x-L_\e)\right]dx\\
 &\qquad \qquad= \int_0^{+\infty}e^{-\lambda x}  \E\left[(1_{\{S_\e<x\}} - 1_{\{L_\e<x\}})f(x-L_\e)\right]dx +  \int_0^{+\infty}e^{-\lambda x}  \E\left[1_{\{L_\e<x\}}f(x-L_\e)\right]dx \\
 &\qquad \qquad \geq f(0)\int_0^{+\infty}e^{-\lambda x}\left( \Pb\left(S_\e<x\right)- \Pb\left(L_\e<x\right)\right)dx +  \int_0^{+\infty}e^{-\lambda x}  \E\left[1_{\{L_\e<x\}}f(x-L_\e)\right]dx\\
  &\qquad \qquad = f(0)\frac{\E\left[e^{-\lambda S_\e}\right]-\E\left[e^{-\lambda L_\e^+}\right]}{\lambda} +  \int_0^{+\infty}e^{-\lambda x}  \E\left[1_{\{L_\e<x\}}f(x-L_\e)\right]dx.
\end{align*}
 We next decompose the remaining integral according as whether $L_\e< 0$ or $L_{\e}\geq 0$ :
\begin{multline*}
 \int_0^{+\infty}e^{-\lambda x}\E\left[ 1_{\{L_\e< x\}}f(x-L_\e)\right] dx\\= \E\left[e^{-\lambda L_\e}1_{\{L_\e\geq0\}}\right]\mathcal{L}[f](\lambda) +   \int_0^{+\infty}e^{-\lambda x}\E\left[ 1_{\{L_\e< 0\}}f(x-L_\e)\right] dx.
\end{multline*}
Then, the Fubini-Tonelli theorem and a change of variable in the last integral yields :
\begin{align*} \int_0^{+\infty}e^{-\lambda x}  \E\left[1_{\{L_\e<0\}}f(x-L_\e)\right] dx &=  \E\left[1_{\{L_\e<0\}} \int_{-L_\e}^{+\infty} e^{-\lambda z - \lambda L_\e } f(z)dz\right] \\
 &\geq  \Pb\left(L_\e<0\right)\mathcal{L}[f](\lambda) -  \E\left[1_{\{L_\e<0\}} \int_{0}^{-L_\e} e^{-\lambda z  } f(z)dz\right]
\end{align*}
and the result follows by gathering the two previous terms :
$$
 \int_0^{+\infty}  e^{-\lambda x}  \E\left[1_{\{L_\e<0\}}f(x-L_\e)\right]dx
 \geq  \E\left[e^{-\lambda L_\e^+}\right]\mathcal{L}[f](\lambda)  -  \E\left[1_{\{L_\e<0\}} \int_{0}^{-L_\e} e^{-\lambda z  } f(z)dz\right].
$$
It remains to compute the limit $iii)$. Notice first that if $\beta=1$ (and thus $\alpha<1$), then  the process $L$ is a subordinator, hence $L_\e \geq0$ a.s and the expectation is null. We thus assume now that $\beta\in(-1,1)$. Let $\varepsilon>0$ and take $A_\varepsilon$ large enough such that $f(x)\leq \varepsilon$  for $x\geq A_\varepsilon$. 
We decompose :
\begin{align*}
& \frac{1}{ \eta_\alpha(\lambda)} \E\left[1_{\{L_\e<0\}} \int_{0}^{-L_\e} e^{-\lambda z  } f(z)dz\right] \\
&\qquad\qquad \leq   \frac{1}{ \eta_\alpha(\lambda)}\E\left[1_{\{-A_\varepsilon<L_\e<0\}} \int_{0}^{A_\varepsilon} e^{-\lambda z  } f(z)dz\right] +  \frac{\varepsilon}{ \eta_\alpha(\lambda)}\E\left[1_{\{L_\e<-A_\varepsilon\}} \int_{A_\varepsilon}^{-L_\e} e^{-\lambda z  } dz\right]\\
&\qquad\qquad \leq  \frac{A_\varepsilon}{ \eta_\alpha(\lambda)} f(0)+  \frac{\varepsilon }{ \eta_\alpha(\lambda)} \int_{0}^{+\infty} e^{-\lambda z  } \Pb(-L_\e>z)dz \xrightarrow[\lambda\downarrow0]{}  \varepsilon\kappa_{\alpha,-\beta}
\end{align*}
where the limit of the integral follows from (\ref{eq:SeLe}) since $-L$ is a stable L\'evy process with parameter $\alpha$ and $-\beta$. 

\end{proof}

We now apply Lemma \ref{lem:f} to study Equation (\ref{eq:u}).
\subsection{Analysis of Equation (\ref{eq:u}) }

\begin{lemma}\label{lem:t1}
The Laplace transform of $\Phi_0 - \Phi_R$  satisfies the following bounds for $\lambda>0$ :
$$ \frac{\lambda}{1-\E[e^{-\lambda S_\e}]}\mathcal{L}[\Phi_0-\Phi_R](\lambda) \leq 1$$
and
$$
    \frac{\lambda}{1-\E\left[e^{-\lambda L_\e^+}\right]}\mathcal{L}[\Phi_0-\Phi_R](\lambda)
 \geq 1- \lambda \mathcal{L}[u](\lambda) -  \frac{\lambda}{1-\E\left[e^{-\lambda L_\e^+}\right]} \E\left[1_{\{L_\e<0\}} \int_{0}^{-L_\e} e^{-\lambda z  } u(z)dz\right].
$$

\end{lemma}

\begin{proof}
Taking the Laplace transform of (\ref{eq:u}) and using Point $i)$ of Lemma \ref{lem:f} with $f=u$, we deduce that 
$$
\mathcal{L}[u](\lambda) \leq  \frac{1-\E\left[e^{-\lambda S_\e}\right]}{\lambda}+ \mathcal{L}[u](\lambda)   \E\left[e^{-\lambda S_\e}\right] -  \mathcal{L}[\Phi_0-\Phi_R](\lambda) $$
which yields the upper bound 
$$
   \frac{\lambda}{1-\E\left[e^{-\lambda S_\e}\right]}
\mathcal{L}[\Phi_0-\Phi_R](\lambda) \leq 1 - \lambda\mathcal{L}[u](\lambda)\leq 1.
$$
To get the lower bound, we apply Point $ii)$ of Lemma \ref{lem:f} still with $f=u$. Since $u(0)=1$, this yields
$$
 \mathcal{L}[u](\lambda)+\mathcal{L}[\Phi_0-\Phi_R](\lambda)  \geq  \frac{1-\E\left[e^{-\lambda L_\e^+}\right]}{\lambda}+ \E\left[e^{-\lambda L_\e^+}\right]\mathcal{L}[u](\lambda) - \E\left[1_{\{L_\e<0\}} \int_{0}^{-L_\e} e^{-\lambda z  } u(z)dz\right]$$
which gives the lower bound, after rearranging the terms.

\end{proof}

It remains now to study the limit of both expressions when $\lambda\downarrow0$.

\subsection{Proof of Theorem \ref{theo:1} when $\alpha\leq 1$}
Notice first that from a change of variable and the monotone convergence theorem,
$$\lambda \mathcal{L}[u](\lambda) = \int_0^{+\infty} e^{-z}u\left(\frac{z}{\lambda}\right) dz \xrightarrow[\lambda\downarrow 0]{} 0.$$
Then letting $\lambda\downarrow0$ in Lemma \ref{lem:t1} and using (\ref{eq:SeLe}) and Point $iii)$ in Lemma \ref{lem:f} with $f=u$, we obtain 
\begin{equation}\label{eq:phia-1}
\mathcal{L}[\Phi_0](\lambda)-\mathcal{L}[\Phi_R](\lambda)\equi_{\lambda\downarrow 0} \kappa_{\alpha, \beta}\, \eta_\alpha(\lambda).
\end{equation}
To simplify the notations, we set 
\begin{equation}\label{eq:varphi0}
\varphi_0(x) =  (1-\E[\boldsymbol{p}])u(x) + \frac{\E[\boldsymbol{p}^2-\boldsymbol{p}]}{2} u^2(x)
\end{equation}
so that
$$\mathcal{L}[\Phi_0](\lambda) = \int_0^{+\infty} e^{-\lambda x} \E\left[1_{\{S_\e<x\}} \varphi_0(x-L_\e)\right] dx.$$
Note that the function $\varphi_0$ is positive and decreasing since $\E[\boldsymbol{p}^2-\boldsymbol{p}]>0$, as $\boldsymbol{p}$ is an integer-valued random variable. 
On the one hand, observe that since $\Phi_R\geq0$, we deduce from Lemma \ref{lem:f} with $f=\varphi_0$ that
$$\mathcal{L}[\Phi_0](\lambda)-\mathcal{L}[\Phi_R](\lambda) \leq \mathcal{L}[\varphi_0](\lambda)$$
hence, from (\ref{eq:phia-1}),
$$\kappa_{\alpha, \beta}\leq  \liminf_{\lambda\downarrow 0}   \frac{1}{\eta_\alpha(\lambda)} \mathcal{L}[\varphi_0](\lambda).$$
On the other hand, fix $\varepsilon>0$ small enough. Since $\lim\limits_{x\rightarrow+\infty} u(x)=0$, there exists $A_\varepsilon>0$ such that $u(x)\leq \varepsilon$ for $x\geq A_\varepsilon$. This implies  that there exists $K$, independent of $\varepsilon$, such that  :
$$u^3(x) \leq  \varepsilon K \varphi_0(x)\qquad \text{for }\quad x\geq A_\varepsilon.$$
As a consequence, using (\ref{eq:phiR}) and Point $i)$ of Lemma \ref{lem:f} with $f=u^3$, we deduce
\begin{align*}
\mathcal{L}[\Phi_R](\lambda) &\leq \E\left[\boldsymbol{p}^3\right]\left( \int_0^{A_\varepsilon} e^{-\lambda x} u^{3}(x) dx + \int_{A_\varepsilon}^{+\infty} e^{-\lambda x} u^{3}(x) dx\right)\\
&\leq   \E\left[\boldsymbol{p}^3\right]\left(A_\varepsilon + \varepsilon K\mathcal{L}[\varphi_0](\lambda)\right).
\end{align*} 
Then, using Point $ii)$ of Lemma \ref{lem:f} with $f=\varphi_0$, 
\begin{multline}\label{eq:infLphi0}
\mathcal{L}[\Phi_0](\lambda)-\mathcal{L}[\Phi_R](\lambda) \geq 
\E\left[e^{-\lambda L_\e^+}\right]\mathcal{L}[\varphi_0](\lambda)+ \varphi_0(0)\frac{\E\left[e^{-\lambda S_\e}\right]-\E\left[e^{-\lambda L_\e^+}\right]}{\lambda}\\   - \E\left[1_{\{L_\e<0\}} \int_{0}^{-L_\e} e^{-\lambda z  } \varphi_0(z)dz\right]-   \E\left[\boldsymbol{p}^3\right]\left(A_\varepsilon + \varepsilon K\mathcal{L}[\varphi_0](\lambda)\right).
\end{multline}
Dividing both sides by $\eta_\alpha(\lambda)$ and applying Point $iii)$ of Lemma \ref{lem:f}, we deduce that 
$$\kappa_{\alpha, \beta}\geq   \left(1-  \varepsilon K \E\left[\boldsymbol{p}^3\right] \right)  \limsup_{\lambda\downarrow 0}   \frac{1}{\eta_\alpha(\lambda)} \mathcal{L}[\varphi_0](\lambda). $$
Finally, we have thus proven that
$$\mathcal{L}[\varphi_0](\lambda) \equi_{\lambda\downarrow0} \kappa_{\alpha,\beta} \,\eta_\alpha(\lambda)$$
hence, by the Tauberian theorem, 
$$\int_0^{x} \varphi_0(z) dz  \equi_{x\rightarrow +\infty} \frac{\kappa_{\alpha,\beta}}{\Gamma(2-\alpha)} \eta_\alpha\left(\frac{1}{x}\right).$$
The result now follows by differentiation, since  $\varphi_0$ is decreasing.

 \qed
\section{{The case $1<\alpha<2$}}\label{sec:3}

In this case, the existence of the first moment of $S_\e$ prevents us from using directly the Tauberian theorem as in the previous case. Indeed, when $\alpha\in(1,2)$, letting $\lambda\downarrow 0$ in the first inequality of Lemma \ref{lem:t1}, one obtains~:
$$\limsup_{\lambda\downarrow0} \mathcal{L}[\Phi_0](\lambda)- \mathcal{L}[\Phi_R](\lambda) \leq \E[S_\e].$$
Going back to  (\ref{eq:infLphi0}), this implies that
\begin{equation}\label{eq:varphi0<inf}
 \left(1-  \varepsilon K \E\left[\boldsymbol{p}^3\right] \right) \limsup_{\lambda\downarrow0} \mathcal{L}[\varphi_0](\lambda)\leq (1+\varphi_0(0))\E[S_\e]+  \E[\boldsymbol{p}^3]A_\varepsilon -\E[ 1_{\{L_\e<0\}}L_\e]\varphi_0(0)<+\infty
 \end{equation}
i.e. we can only deduce, by the monotone convergence theorem, that 
\begin{enumerate}[$i)$]
\item if $\E[\boldsymbol{p}]<1$, 
\begin{equation}\label{eq:u<inf}
\int_0^{+\infty}  u(x) dx =\E[{\bf M}_{\alpha, \beta}]<+\infty,
\end{equation}
\item while if  $\E[\boldsymbol{p}]=1$,
$$
\int_0^{+\infty}  u^2(x) dx <+\infty.
$$
\end{enumerate}
\noindent
We shall thus made a technical modification of the previous proof, and write down some new inequalities. 

\subsection{Preliminary lemma}

\begin{lemma}\label{lem:f2}
Assume that $\alpha\in (1,2)$ and let $f:[0,+\infty)\rightarrow [0,+\infty)$ be a positive and decreasing function. We write $xf$ for the function $x\rightarrow xf(x)$. Then, the following inequalities hold :
\begin{enumerate}[$i)$]
 \item Upper bound :
$$ \int_0^{+\infty}  e^{-\lambda x} x \E\left[1_{\{S_\e<x\}}f(x-L_\e)\right] dx \leq  \E\left[e^{-\lambda S_\e}\right] \mathcal{L}[xf](\lambda) +  \E\left[S_\e e^{-\lambda S_\e}\right] \mathcal{L}[f](\lambda). $$
\item Lower bound :
\begin{multline*}
 \int_0^{+\infty}  e^{-\lambda x}  x\E\left[1_{\{S_\e<x\}}f(x-L_\e)\right]dx\\
 \geq   f(0) \int_0^{+\infty}e^{-\lambda x}x\left( \Pb\left(L_\e\geq x\right)- \Pb\left(S_\e\geq x\right)\right) dx + \int_0^{+\infty}e^{-\lambda x}x  \E\left[1_{\{L_\e<x\}}f(x-L_\e)\right]dx 
\end{multline*}
and
\begin{multline*}
\int_0^{+\infty}e^{-\lambda x}x  \E\left[1_{\{L_\e<x\}}f(x-L_\e)\right]dx 
\\\geq \E\left[e^{-\lambda L_\e^+}\right]\mathcal{L}[xf](\lambda) + \E\left[L_\e e^{-\lambda L_\e^+}\right]\mathcal{L}[f](\lambda)- \E\left[1_{\{L_\e<0\}}\int_0^{-L_\e} e^{-\lambda z}zf(z)dz\right] .
\end{multline*}
\item  Assume that $\lim\limits_{x\rightarrow+\infty}f(x) = 0 $,  then
$$\lim_{\lambda\downarrow0}  \lambda^{2-\alpha}  \E\left[1_{\{L_\e<0\}} \int_{0}^{-L_\e} e^{-\lambda z  } zf(z)dz\right]=0.$$
\end{enumerate}
\end{lemma}
\begin{proof}
The proof of Lemma \ref{lem:f2} is similar to that of Lemma \ref{lem:f}. Point $i)$ follows from the decomposition $x=x-S_\e+S_\e$, using the fact that $f$ is decreasing and computing the Laplace transforms of the convolution products. Point $ii)$ follows similarly from the decomposition $x=x-L_\e +L_\e$, by separating the case $L_\e< 0$ and $L_\e\geq 0$. Finally, for Point $iii)$, observe first that if $\beta=1$, then  the random variable $(-L_\e)^+$ admits exponential moments, hence
$$ \lambda^{2-\alpha} \E\left[1_{\{L_\e<0\}} \int_{0}^{-L_\e} e^{-\lambda z  } zf(z)dz\right] \leq  \lambda^{2-\alpha} f(0) \E\left[ 1_{\{L_\e<0\}}L_\e^2 \right] 
\xrightarrow[\lambda\downarrow0]{} 0.$$
Take now $\beta\in(-1,1)$ and let $\varepsilon>0$. By assumption, there exists $A_\varepsilon>0$ such that $f(x)\leq \varepsilon$  for $x\geq A_\varepsilon$. We then decompose :
\begin{align*}
& \lambda^{2-\alpha} \E\left[1_{\{L_\e<0\}} \int_{0}^{-L_\e} e^{-\lambda z  } zf(z)dz\right] \\
&\qquad \leq   \lambda^{2-\alpha}\E\left[1_{\{-A_\varepsilon<L_\e<0\}} \int_{0}^{A_\varepsilon} e^{-\lambda z  } zf(z)dz\right] + \varepsilon  \lambda^{2-\alpha}\E\left[1_{\{L_\e<-A_\varepsilon\}} \int_{A_\varepsilon}^{-L_\e} e^{-\lambda z  } z dz\right]\\
&\qquad \leq 
 \lambda^{2-\alpha}  A_\varepsilon^2\,f(0)  + \varepsilon  \lambda^{2-\alpha} \int_{0}^{+\infty} e^{-\lambda z  } z\Pb(-L_\e>z)dz\xrightarrow[\lambda\downarrow0]{}  \varepsilon\kappa_{\alpha,-\beta} 
\end{align*}
where the limit of the integral follows as before from (\ref{eq:SeLe}) and the Tauberian theorem. 

\end{proof}

We now apply Lemma \ref{lem:f2} to the equation (\ref{eq:u}) satisfied  by $u$.

\subsection{Analysis of Equation (\ref{eq:u}) in the case $\alpha\in(1,2)$}

\begin{lemma}\label{lem:a+1}
The Laplace transform of $x(\Phi_0-\Phi_R)$ satisfies the following bounds :
$$
\frac{\lambda^2}{1-\E\left[e^{-\lambda S_\e}\right] - \lambda \E\left[S_\e e^{-\lambda S_\e}\right]} \mathcal{L}[x(\Phi_0-\Phi_R)](\lambda)
 \leq 1+ \frac{\lambda^2\, \E\left[S_\e e^{-\lambda S_\e}\right]}{1-\E\left[e^{-\lambda S_\e}\right] - \lambda \E\left[S_\e e^{-\lambda S_\e}\right]} \mathcal{L}[u](\lambda)
$$
and
\begin{multline*}
\frac{\lambda^2}{1-\E\left[e^{-\lambda L_\e^+}\right] - \lambda \E\left[L_\e^+ e^{-\lambda L_\e^+}\right]}\mathcal{L}[x(\Phi_0-\Phi_R)](\lambda)\\ \geq 1-\lambda^2 \mathcal{L}[xu](\lambda)  - \frac{\lambda^2\, \Xi(\lambda) }{1-\E\left[e^{-\lambda L_\e^+}\right] - \lambda \E\left[L_\e^+ e^{-\lambda L_\e^+}\right]}  
\end{multline*}
where 
$$
\Xi(\lambda) =  \E\left[1_{\{L_\e<0\}}\int_0^{-L_\e} e^{-\lambda z}zu(z)dz\right]+ \E\left[1_{\{L_\e<0\}}(-L_\e)\right]   \mathcal{L}[u](\lambda).
$$
\end{lemma}

\begin{proof}
We first multiply (\ref{eq:u}) by $x$ before taking the Laplace transform of both sides:
$$ \mathcal{L}[xu](\lambda)= 
  \int_0^{+\infty}e^{-\lambda x}x\Pb\left(S_\e\geq x\right)dx +\int_0^{+\infty}e^{-\lambda x}x  \E\left[1_{\{S_\e<x\}}u(x-L_\e)\right]dx - \mathcal{L}[x(\Phi_0-\Phi_R)](\lambda).
$$
Integrating by parts the first term on the right-hand side, we have 
\begin{equation}\label{eq:xP}
 \int_0^{+\infty}e^{-\lambda x}x\Pb\left(S_\e\geq x\right)dx=    \frac{1-\E\left[e^{-\lambda S_\e}\right] - \lambda \E\left[S_\e e^{-\lambda S_\e}\right]}{\lambda^2}\equi_{\lambda\downarrow0} \kappa_{\alpha, \beta} \Gamma(2-\alpha) \lambda^{\alpha-2}  
 \end{equation}
from (\ref{eq:SeLe}) and the Tauberian theorem.
To get the upper bound, we apply Point $i)$ of Lemma \ref{lem:f2} with $f=u$ :
\begin{multline*}
\mathcal{L}[x(\Phi_0-\Phi_R)](\lambda)+ \mathcal{L}[xu](\lambda) \\
 \leq \frac{1-\E\left[e^{-\lambda S_\e}\right]-\lambda\E\left[S_\e e^{-\lambda S_\e}\right]}{\lambda^2} +\E\left[e^{-\lambda S_\e}\right]\mathcal{L}[xu](\lambda) +\E\left[S_\e\, e^{-\lambda S_\e}\right]\mathcal{L}[u](\lambda).
 \end{multline*}
Adding $\lambda \E[S_\e e^{-\lambda S_\e}] \mathcal{L}[xu](\lambda)$ on the right-hand side and rearranging the terms  yields the announced upper bound.\\

\noindent
Similarly, to get the lower bound, we apply Point $ii)$ of Lemma \ref{lem:f2} with $f=u$ :
\begin{multline*}
\mathcal{L}[x(\Phi_0-\Phi_R)](\lambda)+ \mathcal{L}[xu](\lambda) \\
 \geq \frac{1-\E\left[e^{-\lambda L_\e^+}\right]-\lambda\E\left[L_\e^+ e^{-\lambda L_\e^+}\right]}{\lambda^2} +\E\left[e^{-\lambda L_\e^+}\right]\mathcal{L}[xu](\lambda) \\
 +\E\left[L_\e\, e^{-\lambda L_\e^+}\right]\mathcal{L}[u](\lambda)
  - \E\left[1_{\{L_\e<0\}} \int_0^{-L_\e} e^{-\lambda z}z u(z) dz  \right] 
 \end{multline*}
Adding $\lambda  \E\left[L_\e^+ e^{-\lambda L_\e^+}\right] \mathcal{L}[xu](\lambda)$ in both sides and rearranging the terms, we obtain 
\begin{multline*}
\frac{\lambda^2}{1-\E\left[e^{-\lambda L_\e^+}\right] - \lambda \E\left[L_\e^+ e^{-\lambda L_\e^+}\right]}\mathcal{L}[x(\Phi_0-\Phi_R)](\lambda)\\
 \geq 1- \lambda^2 \mathcal{L}[xu](\lambda) - \frac{\lambda^2 }{1-\E\left[e^{-\lambda L_\e^+}\right] - \lambda \E\left[L_\e^+ e^{-\lambda L_\e^+}\right]}R(\lambda) 
\end{multline*}
where the remainder $R(\lambda)$ is given by :
\begin{multline*}
R(\lambda)=\E\left[1_{\{L_\e<0\}} \int_0^{-L_\e} e^{-\lambda z}z u(z) dz  \right] -  \E\left[1_{\{L_\e<0\}}L_\e\right]\mathcal{L}[u](\lambda)\\+  \E\left[L_\e^+ e^{-\lambda L_\e^+}\right]\left(\lambda \mathcal{L}[xu](\lambda)-\mathcal{L}[u](\lambda)\right). 
\end{multline*}
Note that the last term on the right-hand side is negative and may thus be removed since, integrating by parts, 
$$\mathcal{L}[u](\lambda) - \lambda \mathcal{L}[xu](\lambda)= \E\left[{\bf M}_{\alpha,\beta}e^{-\lambda {\bf M}_{\alpha, \beta}}\right]\geq0,$$
hence  $R(\lambda)\leq \Xi(\lambda)$ which proves the lower bound.

\end{proof}

\subsection{Proof of Theorem \ref{theo:1} when $\alpha>1$}

We now want to let $\lambda\downarrow0$ in Lemma \ref{lem:a+1}. Notice first that, thanks to (\ref{eq:SeLe}), the first terms on the left-hand side of both inequalities will converge towards the same quantity, which is given by (\ref{eq:xP}).
 Also, by the monotone convergence theorem
$$\lambda^2 \mathcal{L}[xu](\lambda) = \int_0^{+\infty} e^{-z} z u\left(\frac{z}{\lambda}\right) dz \xrightarrow[\lambda\downarrow0]{}0$$
and, applying Point $iii)$ of Lemma \ref{lem:f2} with $f=u$, 
$$\lim_{\lambda\downarrow0}  \lambda^{2-\alpha}  \E\left[1_{\{L_\e<0\}} \int_{0}^{-L_\e} e^{-\lambda z  } z u(z)dz\right]=0.$$
As a consequence, it only remains to show that 
\begin{equation}\label{eq:lu=0}
\lim\limits_{\lambda\downarrow0}\lambda^{2-\alpha} \mathcal{L}[u](\lambda) =0.
\end{equation}
When $\E[\boldsymbol{p}]<1$, this is a direct consequence of (\ref{eq:u<inf}) since  $ \mathcal{L}[u](\lambda) \leq \E[{\bf M}_{\alpha,\beta}]<+\infty$. The situation is trickier when $\E[\boldsymbol{p}]=1$, and we shall rely on the following Lemma, which gives an a priori bound on $u$.

\begin{lemma}\label{lem:M}
Assume that $\alpha>1$ and $\E[\boldsymbol{p}]=1$. Then,  there exists a constant $C>0$ such that for $x$ large enough, 
$$\Pb\left({\bf M}_{\alpha,\beta}\geq x\right) \leq C x^{-\alpha/2}.$$
\end{lemma}

\begin{proof}
We mimic the arguments of \cite{LaSh2}. Recall that since $\alpha>1$, the positivity parameter $\rho = \Pb(L_1\geq 0)$ of $L$ belongs to the interval $[1-\frac{1}{\alpha}, \, \frac{1}{\alpha}]$. Denote by $\underline{\bf M}_{\alpha,\beta}^{(t)}$ the maximum of the branching stable process on the interval $[0,t]$ and by $\overline{\bf M}_{\alpha,\beta}^{(t)}$ the maximum of the branching stable process on the interval $[t,+\infty]$. From the construction given in (\ref{eq:Z}), we first deduce that 
\begin{align*}
\Pb\left({\bf M}_{\alpha,\beta}\geq x\right) &\leq \Pb\left(\underline{\bf M}_{\alpha,\beta}^{(t)}\geq x\right)+\Pb\left(\overline{\bf M}_{\alpha,\beta}^{(t)}\geq x\right)\\
&\leq \Pb\left(\underline{\bf M}_{\alpha,\beta}^{(t)}\geq x\right) + \Pb\left(Z(t)\geq 1\right).
\end{align*}
Let us now denote by $T_x$ the first time at which a particle of the branching process reaches the level $x$.
Then, by applying the Markov property at $T_x$, we deduce that conditionally to $\{T_x\leq t\} =\{\underline{\bf M}_{\alpha,\beta}^{(t)}\geq x\} $, the expected number of particles above $x$ at time $t$ is greater than $\rho$ :
$$\E\left[\sum_{i=1}^{Z(t)}  1_{\{L_t^{(i)}\geq x\}} \bigg| \underline{\bf M}_{\alpha,\beta}^{(t)}\geq x\right]\geq \rho\geq 1-\frac{1}{\alpha}$$
which implies that
\begin{align*}
\Pb\left(\underline{\bf M}_{\alpha,\beta}^{(t)}\geq x\right)&\leq \frac{\alpha}{\alpha-1} \E\left[\sum_{i=1}^{Z(t)}  1_{\{L_t^{(i)}\geq x\}} \right]\leq \frac{\alpha}{\alpha-1} \E[Z(t)]\Pb\left(L_t\geq x\right)= \frac{\alpha}{\alpha-1} \Pb\left( t^{1/\alpha}L_1\geq x\right)
\end{align*}
since $Z$ is independent from the positions $(L^{(i)})$, and $\E[Z(t)]=1$ for all $t\geq0$. Taking $t=x^{\alpha/2}$, we have thus proven that 
$$\Pb\left({\bf M}_{\alpha,\beta}\geq x\right)\leq  \frac{\alpha}{\alpha-1} \Pb\left(L_1\geq \sqrt{x}\right) +  \Pb\left(Z(x^{\alpha/2})\geq 1\right)
$$
and the result follows by using  (\ref{eq:SeLe}) and the Kolmogorov's theorem, which states that
$$\Pb\left(Z(t)\geq 1\right) \equi_{t\rightarrow+\infty} \frac{2}{\sigma^2 t}$$
 see for instance Asmussen \& Hering \cite[Theorem 2.6]{AsHe}.

\end{proof}

We now come back to the limit (\ref{eq:lu=0}). Applying Lemma \ref{lem:M} and the Tauberian theorem, we deduce that for $\lambda$ small enough, there exists a constant $C>0$ such that 
$$
\lambda^{2-\alpha}\mathcal{L}[u](\lambda) \leq C \lambda^{1-\frac{\alpha}{2}} \xrightarrow[\lambda\downarrow0]{}0
$$
since $\alpha\in(1,2)$. As a consequence, by letting $\lambda\downarrow0$ in Lemma \ref{lem:a+1}, we obtain the asymptotics
$$ \mathcal{L}[x\Phi_0](\lambda)-\mathcal{L}[x\Phi_R](\lambda) \equi_{\lambda\downarrow0} \kappa_{\alpha, \beta}\Gamma(2-\alpha)\lambda^{\alpha-2}.$$
The remainder of the proof is now similar to the case $\alpha \in (0,1]$. First, applying Point $i)$ of Lemma \ref{lem:f2} with $f=\varphi_0$ and using that $\Phi_R\geq0$, we deduce that 
$$\mathcal{L}[x\Phi_0](\lambda)-\mathcal{L}[x\Phi_R](\lambda) \leq \mathcal{L}[x\varphi_0](\lambda)+  \E[S_\e] \mathcal{L}[\varphi_0](\lambda)
$$
which implies the  lower bound
$$\liminf_{\lambda\downarrow 0} \lambda^{2-\alpha}\mathcal{L}[x\varphi_0](\lambda) \geq \kappa_{\alpha, \beta} \Gamma(2-\alpha). $$

\noindent
Next, we deduce from Point $ii)$ of Lemma \ref{lem:f2} with $f=\varphi_0$ that
\begin{multline*}
\mathcal{L}[x\Phi_0](\lambda)\\\geq  \varphi_0(0) \int_0^{+\infty}e^{-\lambda x}x\left( \Pb\left(L_\e\geq x\right)- \Pb\left(S_\e\geq x\right)\right) dx +  \int_0^{+\infty}e^{-\lambda x}x  \E\left[1_{\{L_\e<x\}}\varphi_0(x-L_\e)\right]dx 
\end{multline*}
where 
\begin{multline*}
\int_0^{+\infty}e^{-\lambda x}x  \E\left[1_{\{L_\e<x\}}\varphi_0(x-L_\e)\right]dx 
\\\geq \E\left[e^{-\lambda L_\e^+}\right]\mathcal{L}[x\varphi_0](\lambda) - \E\left[1_{\{L_\e<0\}}\int_0^{-L_\e} e^{-\lambda z}z\varphi_0(z)dz\right] + \E\left[L_\e e^{-\lambda L_\e^+}\right]\mathcal{L}[\varphi_0](\lambda) .
\end{multline*}
Take $\varepsilon>0$. As before, since $\lim\limits_{x\rightarrow+\infty} u(x)=0$, there exists $A_\varepsilon>0$ and $K>0$  such that 
$$u(x)\leq \varepsilon\qquad \text{and}\qquad u^3(x) \leq  \varepsilon K \varphi_0(x)\qquad \text{for }\quad x\geq A_\varepsilon.$$
As a consequence, using (\ref{eq:phiR}) and Point $i)$ of Lemma \ref{lem:f2} with $f=u^3$, we deduce
\begin{align*}
\mathcal{L}[x\Phi_R](\lambda) &\leq \E\left[\boldsymbol{p}^3\right] \left( \mathcal{L}[xu^3](\lambda) +  \E\left[S_\e\right] \mathcal{L}[u^3](\lambda)\right) \\
&\leq   \E\left[\boldsymbol{p}^3\right]\left(A^2_\varepsilon + \E\left[S_\e\right] A_\varepsilon + \varepsilon K\left(\mathcal{L}[x\varphi_0](\lambda) +\E\left[S_\e\right]\mathcal{L}[\varphi_0](\lambda) \right) \right)\\
&\leq C (1 + \varepsilon \mathcal{L}[x\varphi_0](\lambda))
\end{align*} 
for some constant $C$ large enough since $\mathcal{L}[\varphi_0](\lambda) \leq \mathcal{L}[\varphi_0](0)<+\infty$ from (\ref{eq:varphi0<inf}).
Then, using Point $iii)$ of Lemma \ref{lem:f2} with $f=\varphi_0$, we obtain
$$(1- \varepsilon C ) \limsup_{\lambda\downarrow0} \mathcal{L}[x\varphi_0](\lambda)\leq  \kappa_{\alpha, \beta} \Gamma(2-\alpha)$$
which implies that
$$\mathcal{L}[x\varphi_0](\lambda) \equi_{\lambda\downarrow0}\kappa_{\alpha, \beta} \Gamma(2-\alpha) \lambda^{\alpha-2}.$$
Finally, by the Tauberian theorem, 
$$\int_0^{x} x\varphi_0(z) dz  \equi_{x\rightarrow +\infty} \frac{\kappa_{\alpha,\beta}}{2-\alpha} x^{2-\alpha} $$
and the result now follows by differentiation, since  $\varphi_0$ is decreasing, see Lemma \ref{lem:app} in the Appendix.\\

\qed

\section{Appendix}

We briefly prove the following lemma, which allows to differentiate an asymptotics :

\begin{lemma}\label{lem:app}
Assume that $u$ is a positive and decreasing function and $\gamma\in(0,1)$. Then, for any $\xi\geq0$ :
$$\int_0^{x} z^\xi u(z) dz \equi_{x\rightarrow +\infty} x^\gamma \qquad \Longleftrightarrow \qquad  u(x) \equi_{x\rightarrow +\infty} \gamma x^{\gamma-1-\xi}$$
\end{lemma}

\begin{proof}
We only prove the implication $\Longrightarrow$ as its converse is classic. Since $u$ is decreasing, we have for $h>0$ :
$$\int_x^{x+h} z^\xi u(z)dz \leq  h (x+h)^\xi u(x).$$
Fix $\varepsilon>0$. Assuming that $x$ is large enough, we deduce from the assumption that 
$$  (x+h)^\gamma-x^\gamma -2\varepsilon((x+h)^\gamma + x^\gamma) \leq h (x+h)^\xi u(x).$$
Dividing both sides by $h$, we further obtain since $\gamma\in(0,1)$ :
$$\gamma (x+h)^{\gamma-1}  -\frac{2\varepsilon((x+h)^\gamma + x^\gamma)}{h}\leq  (x+h)^\xi u(x).$$
We now set $h=\sqrt{\varepsilon}x$, divide both sides by $\gamma x^{\gamma-1}$ and let $x\rightarrow +\infty$. This yields :
$$(1+\sqrt{\varepsilon})^{\gamma-1} - 2\sqrt{\varepsilon}\frac{(1+\sqrt{\varepsilon})^\gamma+1}{\gamma}\leq \left(1+\sqrt{\varepsilon}\right)^\xi
\liminf_{x\rightarrow+\infty}\frac{x^\xi u(x)}{\gamma x^{\gamma-1}}$$
i.e., letting $\varepsilon\downarrow0$, 
$$1\leq
\liminf_{x\rightarrow+\infty}\frac{x^\xi u(x)}{\gamma x^{\gamma-1}}.$$
The other bound is obtained similarly, by starting from $\displaystyle \int_{x-h}^{x} z^\xi u(z)dz$.

\end{proof}

\end{document}